\documentclass[11pt]{amsart}
\usepackage{amssymb,amsmath,amsthm}
\usepackage{enumitem}
\usepackage{tikz-cd}
 \usepackage{mathtools}
\usetikzlibrary{matrix,arrows,decorations.pathmorphing,shapes.geometric, calc}
\usepackage{verbatim}
\usepackage{psfrag}
\usepackage{caption}
\usepackage{subcaption}
\usepackage{booktabs}
 \usepackage{multirow}	

\newtheorem{conjecture}{Conjecture}
\newtheorem{theorem}[equation]{Theorem}
\newtheorem{lemma}[equation]{Lemma}
\newtheorem{prop}[equation]{Proposition}
\newtheorem{cor}[equation]{Corollary}
\newtheorem{corollary}[equation]{Corollary}

\newtheorem{definition}[equation]{Definition}

\theoremstyle{remark}
\newtheorem{remark}[equation]{Remark}

\newtheorem{assumption}[equation]{Assumption}

\numberwithin{equation}{section}

\definecolor{light-gray}{gray}{.95}

\newcommand{\Sph}{\mathbb{S}}
\newcommand{\Spheq}{\mathbb{S}^2_{\mathrm{eq}}}
\newcommand{\Rcap}{\mathsf{R}}
\newcommand{\Rcapunder}{{\mathsf{R}}}
\newcommand{\Ecal}{\mathcal{E}}

\newcommand{\N}{\mathbb{N}}

\newcommand{\R}{\mathbb{R}}

\newcommand{\B}{\mathbb{B}}

\newcommand{\Ncal}{\mathcal{N}}

\newcommand{\Coord}{\mathcal{C}}
\newcommand{\Acal}{\mathcal{A}}
\newcommand{\Scal}{\mathcal{S}}
\newcommand{\Vcal}{\mathcal{V}}
\newcommand{\group}{G}
\newcommand{\groupref}{G_{\mathcal{P}}}

\newcommand{\RP}{\mathbb{R}P}

\newcommand{\firsteigen}{\mathcal{E}_{\sigma_1}}
\newcommand{\Ecali}{\mathcal{E}_{\sigma_i}}

\newcommand{\rhotilde}{\widetilde{\rho}}

\newcommand{\Vs}{\widetilde{\mathcal{C}}}
\newcommand{\kcir}{k_{\mathrm{cir}}}

\begin{document}

\title[]{On Steklov Eigenspaces for Free Boundary Minimal Surfaces in the Unit Ball}
\author[R.~Kusner]{Robert~Kusner}
\author[P.~McGrath]{Peter~McGrath}
\date{}
\address{Department of Mathematics, University of Massachusetts,
Amherst, MA, 01003} \email{profkusner@gmail.com, kusner@umass.edu}
\address{Department of Mathematics, North Carolina State University, Raleigh NC 27695} 
\email{pjmcgrat@ncsu.edu}

\begin{abstract}
We develop new methods to compare the span $\Coord(\Sigma)$ of the coordinate functions on a free boundary minimal submanifold $\Sigma$ embedded in the unit $n$-ball $\B^n$ with its first Steklov eigenspace $\firsteigen(\Sigma)$.  Using these methods, we show that $\Coord(A)=\firsteigen(A)$ for any embedded free boundary minimal annulus $A$ in $\B^3$ invariant under the antipodal map, and thus prove that $A$ is congruent to the critical catenoid.  We also confirm that $\Coord=\firsteigen$ for any free boundary minimal surface embedded in $\B^3$ with the symmetries of many known or expected examples, including: examples of any positive genus from stacking at least three disks; two infinite families of genus $0$ examples with dihedral symmetry, as well as a finite family with the various Platonic symmetries; and examples of any genus by desingularizing several disks that meet at equal angles along a diameter of the ball.
\end{abstract}
\maketitle

\section{Introduction}
\label{S:intro}

Geometers have made tremendous progress  \cite{FraserSurvey,LiSurvey} in developing existence theory for free boundary minimal surfaces in the last decade.  The rich variety of methods employed include variational techniques \cite{FraserSchoen, Maximo}, singular perturbation gluing methods \cite{Zolotareva, KapLi, KapWiygul}, and min-max constructions \cite{Carlotto, DeLellis, LiZhou, Ketover}.  On the other hand, uniqueness results for free boundary minimal surfaces in the Euclidean unit ball $\B^n$ are rare. 
 Although Nitsche \cite{Nitsche} and Fraser-Schoen~\cite{FSdisk} have used Hopf-differential methods to 
 show the only free boundary minimal disk in $\B^n$ is congruent to a flat equator, the important and challenging problem of classifying free boundary minimal surfaces remains open, even for annuli: 

\begin{conjecture}[Fraser-Li, \cite{FraserLi}]
\label{Ccc} 
Up to congruence, the critical catenoid is the only properly embedded free boundary minimal annulus in $\B^3$.   
\end{conjecture}

Fraser-Schoen established \cite{FraserSchoen} this conjecture provided the first Steklov eigenvalue $\sigma_1$ of the annulus equals $1$, and it was later shown \cite{McGrath} that $\sigma_1=1$ holds assuming the annulus is symmetric with respect to reflections through the coordinate planes.  A key result of this paper reduces the symmetry assumption as far as possible short of the full conjecture: 

\begin{theorem}
\label{Tcc}
Let $A$ be an embedded, antipodally-invariant free boundary minimal annulus in $\B^3$.  Then $A$ is congruent to the critical catenoid.
\end{theorem}

The technique used to prove Theorem \ref{Tcc} is a special case of a general method developed in this paper for comparing the span $\Coord(\Sigma)$ of the coordinate functions on a free boundary minimal submanifold $\Sigma$ in $\B^n$ with its first Steklov eigenspace $\firsteigen(\Sigma)$; these coordinate functions $\Coord(\Sigma)$ are Steklov eigenfunctions with eigenvalue $1$, and Fraser-Li conjectured the following:
 
 \begin{conjecture}[Fraser-Li, \cite{FraserLi}]
 \label{CFL}
 Let $\Sigma$ be a properly embedded minimal hypersurface in the Euclidean unit ball $\B^n$ with free boundary on $\partial \B^n$.  Then the first Steklov eigenvalue $\sigma_1$ of $\Sigma$ equals $1$. 
 \end{conjecture}
 \noindent

Analogously, the coordinate functions on an $n$-dimensional minimal hypersurface in $\Sph^{n+1}$ are Laplace eigenfunctions with eigenvalue $n$, so Conjecture \ref{CFL} is the analog of a well-known conjecture of Yau \cite{Yau:Problems}:
 \begin{conjecture}[Yau, \cite{Yau:Problems}]
 \label{Cyau}
 Let $\Sigma$ be a closed embedded minimal hypersurface of the round unit sphere $\Sph^{n+1}$.  Then the first eigenvalue $\lambda_1$ of the Laplacian on $\Sigma$ equals $n$. 
 \end{conjecture}  
\noindent

The first general evidence for Yau's conjecture was provided by Choi-Wang \cite{ChoiWang}, who proved $\lambda_1\geq \frac{n}{2}$.  Their method was used by Fraser-Li \cite{FraserLi} to prove $\sigma_1 \geq \frac{1}{2}$ in the free boundary case, in support of Conjecture \ref{CFL}.  Choe-Soret \cite{ChoeSoret} verified Conjecture \ref{Cyau} for a large number of minimal surfaces in $\Sph^3$, including the Lawson surfaces \cite{Lawson} and the Karcher-Pinkall-Sterling surfaces \cite{KPS}.  Their method applies to minimal surfaces invariant under certain symmetry groups generated by reflections through great two-spheres and was later adapted to the free boundary minimal setting \cite{McGrath}.  In the setting of $\Sph^3$, it was observed \cite{KWW} that the Choe-Soret method can be modified to show the first Laplacian eigenspace $\Ecal_{\lambda_1}$ is spanned by the four coordinate functions on these examples; this led those authors to propose the following strengthening of Yau's Conjecture:
\begin{conjecture}[\cite{KWW}]
The first Laplacian eigenspace $\Ecal_{\lambda_1}(\Sigma)$ on each closed embedded minimal hypersurface of the unit sphere $\Sph^{n+1}$ coincides with the span $\Coord(\Sigma)$ of its coordinate functions.
\end{conjecture}

 In light of our work in this paper, we propose to strengthen Conjecture \ref{CFL} as follows:
\begin{conjecture}
The first Steklov eigenspace $\firsteigen(\Sigma)$ on each properly embedded minimal hypersurface $\Sigma\subset\B^n$ with free boundary on $\partial \B^n$ coincides with the span $\Coord(\Sigma)$ of its coordinate functions.
\end{conjecture}

In recent years, constructions of free boundary minimal surfaces have emerged for which the methods of \cite{McGrath} do not apply.  In particular, surfaces like the Kapouleas-Li desingularizations  \cite{KapLi} of the critical catenoid and an equatorial disk, or Kapouleas-Wiygul's  triplings \cite{KapWiygul} of the equatorial disk, have symmetries---like halfturns about lines or antipodal symmetries---beyond those considered in \cite{McGrath}.

\subsection*{Brief discussion of the results}
\phantom{ab}
\nopagebreak

Our main result, Theorem \ref{Tgroup}, shows that $\firsteigen = \Coord$ when $\Sigma$ is an embedded free boundary minimal surface in $\B^3$, provided $\Sigma$ is invariant under a group of isometries $G$ satisfying certain technical properties.  Theorem \ref{Tgroup} subsumes the results in \cite{McGrath}, applies to families of known examples outside the scope of \cite{McGrath}, and also to any surfaces with the symmetries of families expected to exist, including:

\begin{itemize} 
\item Stackings of $k\geq 2$ copies of the equatorial disk, including triplings of any genus \cite{KapWiygul} and doublings \cite{Zolotareva}.
\item Genus zero examples: a finite family with the symmetries of the various Platonic solids \cite{Ketover, GL, KOO}, and infinite families with dihedral symmetry---doublings from \cite{Zolotareva}, another family with an additional pair of opposed vertical ends, and families from \cite{KapZou}.
\item Examples of any genus desingularizing $m\geq 2$ disks meeting at equal angles along a diameter of $\B^3$ \cite{KapLi}.
\end{itemize}

In the setting of embedded minimal surfaces in $\Sph^3$, we prove an analogous result, Theorem \ref{Tsphere}, which generalizes \cite[Theorem 4.2]{ChoeSoret} and applies to examples outside of the scope of \cite{ChoeSoret} such as doublings of the equatorial sphere~\cite{McGrathKap1}.

\subsection*{Brief discussion of the methods}
\phantom{ab}
\nopagebreak

The new methods in this paper are based on the two-piece property for free boundary minimal hypersurfaces in the unit ball \cite{LimaMenezes},
and apply to surfaces invariant under symmetry groups $G$ generated by reflections through general subspaces $V$ of $\R^3$, including (for the first time) zero- and one-dimensional subspaces.

The basic idea, following earlier work \cite{ChoeSoret, McGrath}, 
is to decompose $\firsteigen$ into a sum of subspaces, each consisting of eigenfunctions that are either even or odd with respect to generators for $G$.

With suitable topological assumptions about the quotient orbifold $\Sigma/G$, we are able to show (see Proposition \ref{Psym}) using the Courant nodal domain theorem that the space $\firsteigen^G$ of $G$-invariant elements of $\firsteigen$ is trivial. 

For each generator $\rho$ of $G$, we next consider the subspace $\Acal_\rho( \firsteigen)$ spanned by eigenfunctions satisfying $u \circ \rho = - u$, with the aim of showing that $\Acal_\rho(\firsteigen)$ is a subspace of the set of coordinate functions $\Coord_{V^\perp}$ vanishing on the fixed-point set $V$ of $\rho$.  The case where $V$ is a hyperplane is already well-understood \cite{McGrath, ChoeSoret}, and for each nonzero $u \in \Acal_{\rho}( \firsteigen)$ it follows (see \cite[Lemma 3.2]{McGrath}  and Corollary \ref{c} below) that $\int_{\partial \Sigma} u \varphi \neq 0$ for each nonzero $\varphi \in \Coord_{V^\perp}$.

The case where $V$ has higher codimension---for example, where $V = \{0\}$, $\Coord_{V^\perp} = \Coord$, and $\rho$ is the antipodal map---is more subtle because $V$ does not separate $\R^n$, and it is not generally true that $\int_{\partial \Sigma} u \varphi \neq 0$ for $\varphi \in \Coord_{V^\perp}$. However, in certain cases, we use the two-piece property to nonetheless align along the boundary the nodal domains of $u$ and a judiciously chosen $\varphi \in \Coord_{V^\perp}$ to conclude that $\int_{\partial \Sigma} u \varphi \neq 0$ and thus that $u$ and $\varphi$ lie in the same eigenspace. 

In the setting of either Theorem \ref{Tcc} or \ref{Tgroup}, we use the preceding ideas to prove $\firsteigen = \Coord$.  In the former case---when $\Sigma$ is an antipodally-symmetric free boundary minimal annulus---we appeal to a result of Fraser-Schoen \cite{FraserSchoen} to conclude that $\Sigma$ is the critical catenoid.  Their argument---which works in any codimension---parametrizes $\Sigma$ by a conformal harmonic map $\Phi$ from the round annulus to $\B^n$ and exploits $\sigma_1 = 1$ to show its angular derivative $\Phi_\theta$ extends to a rotational Killing field on $\B^n$, hence that $\Sigma$ is rotationally symmetric. 

Interestingly, we are not able to exploit reflectional symmetries through general dimensional subspaces $V$ in the setting of closed minimal surfaces in $\Sph^3$.  Although the analogous two-piece property holds \cite{Ros}, in order to parallel the argument outlined above, the nodal domains of a nonzero $u \in \Acal_\rho(\Ecal_{\lambda_1})$ and an appropriate $\varphi \in \Coord_{V^\perp}$ would have to agree on $\Sigma$---rather than on $\partial \Sigma$ as above---in order to conclude $\int_{\Sigma} u \varphi \neq 0$.

\subsection*{Acknowledgments}
$\phantom{ab}$
\nopagebreak

We are grateful to Ailana Fraser for her interest and for several helpful discussions.  We also thank Nikos Kapouleas for enlightenment about desingularizations of disks meeting in a common diameter, and W{\"o}den Kusner for pointing out how Corollary \ref{Cfs}(i) implies \ref{Cfs}(ii).  Finally, we thank the referee, whose detailed and comprehensive suggestions improved the paper. 

\section{The Steklov Problem and Riemannian Symmetries}
\label{Smain}
Let $(\Sigma, \partial \Sigma)$ be a smooth, compact, connected Riemannian $m$-manifold with boundary.  Let $\eta$ be the unit outward pointing conormal vector field on $\partial \Sigma$.  The \emph{Steklov eigenvalue problem} is
\begin{equation}
\label{Esteklov}
  \begin{cases} 
      \hfill \Delta u  = 0    \hfill & \text{in}\quad \Sigma \\
      \hfill \frac{\partial u}{\partial \eta} = \sigma u \hfill & \text{on}\quad \partial \Sigma \\
  \end{cases}
\end{equation}
and we call a nontrivial harmonic function $u\in\mathcal{H}(\Sigma)$ satisfying \eqref{Esteklov} a \emph{Steklov eigenfunction}.  
The eigenvalues of \eqref{Esteklov} are the spectrum of the \emph{Dirichlet-to-Neumann map} $\mathcal{D}: C^\infty(\partial \Sigma)\cong\mathcal{H}(\Sigma) \rightarrow C^{\infty}(\partial \Sigma)\cong\mathcal{H}(\Sigma)$ given by
\[ \mathcal{D} u = \frac{\partial u}{\partial \eta} \]
where we have identified the harmonic extension of $u$ to $\Sigma$ with its boundary values.  With respect to the $L^2(\partial \Sigma)$ inner product, $\mathcal{D}$ is a self-adjoint pseudodifferential operator with discrete spectrum
\[ 0 = \sigma_0 < \sigma_1 \leq \sigma_2 \leq \cdots. \]
We denote the $\mathcal{D}$-eigenspace corresponding to $\sigma_i$ by $\Ecal_{\sigma_i}$.

The \emph{nodal set} of an eigenfunction $u$ is $ \Ncal_u := \{ p \in \Sigma  : u(p) = 0\}$ and
a \emph{nodal domain} of $u$ is a connected component of $\Sigma \setminus \Ncal_u$.  
 The following Courant-type nodal domain theorem is standard  \cite{Proc}:
\begin{lemma}
\label{Lcourant}
Each nonzero $u \in \firsteigen$ has exactly two nodal domains. 
\end{lemma}

Although we do not carry out the details, many of the results of this section hold with only minor modifications if $\Sigma$ is a closed manifold and \eqref{Esteklov} is replaced with the Laplace eigenvalue problem $\Delta u + \lambda u=0$.  For later use (see  Theorem \ref{Tsphere}), we call such a function $u$ a Laplace eigenfunction, recall the Laplacian $\Delta$ has a discrete spectrum $0 = \lambda_0 < \lambda_1 \leq \lambda_2 \leq \cdots$, and denote the eigenspace corresponding to $\lambda_i$ by $\Ecal_{\lambda_i}$.

\subsection*{Involutive isometries and even-odd decompositions}
$\phantom{ab}$
\nopagebreak

Suppose now that $\Sigma$ admits an involutive isometry $\rho$. 
 Then $\rho$ induces a linear involutive isometry of each $\Ecal_{\sigma_i}$ by the map $u \mapsto u \circ \rho$ and hence $\Ecali$ has an orthogonal direct sum decomposition $\Ecali = \Acal_\rho (\Ecali)\oplus \Scal_\rho(\Ecali)$ into odd and even parts
\[ \Acal_\rho(\Ecali) := \{u \in \Ecali: u\circ \rho =-u\} \, \, \,  
\text{and} \, \, \, 
\Scal_\rho(\Ecali) := \{ u \in \Ecali : u \circ \rho = u\}.
\]

More generally,  we call a function $u$ even (odd) under an isometry $\rho$ if $u \circ \rho = u$ ($u \circ \rho = -u$) and call a set of functions $\rho$-even ($\rho$-odd) if each element is even (odd) under $\rho$.  The following \emph{boundary nodal domain principle} will be useful later for characterizing $\Acal_{\rho}(\firsteigen)$.
\begin{lemma}
\label{Lsym}
Suppose $\Sigma$ admits an involutive isometry $\rho$ and there are Steklov eigenfunctions $u$ and $\varphi$ satisfying the following properties:
\begin{enumerate}[label=\emph{(\alph*)}]
\item $u$ and $\varphi$ are $\rho$-odd; 
\item $u$ has exactly two nodal domains;

\item $u$ and $\varphi$ have nodal domains $\Omega$ and $\Omega_\varphi$ such that $\partial \Sigma \cap \Omega=\partial \Sigma \cap \Omega_\varphi$ $\partial \Sigma$-almost everywhere. 
 
\end{enumerate}
Then $\int_{\partial \Sigma} u \varphi \neq 0$, so $u$ and $\varphi$ are in the same eigenspace. 
\end{lemma}
\begin{proof}
Since $u\circ \rho = -u$, the nodal domains of $u$ are $\Omega$ and $\rho(\Omega)$.
Thus, 
\begin{align*}
\int_{\partial \Sigma} \varphi   u  &= \int_{\partial \Sigma \cap \Omega} \varphi  u  + \int_{\partial \Sigma \cap \rho (\Omega)} \varphi  u   = 2\int_{\partial \Sigma \cap\Omega} \varphi u.
\end{align*} 
But $u$ has a sign on $\Omega$, $\varphi$ has a sign on $\Omega_\varphi$, and $\partial \Sigma \cap \Omega$ coincides with $ \partial \Sigma \cap \Omega_\varphi$ almost everywhere on $\partial\Sigma$, so $\int_{\partial \Sigma} \varphi  u \neq 0$.
\end{proof}

\subsection*{Group invariant eigenfunctions and their nodal domains}
$\phantom{ab}$
\nopagebreak

 If $G$ is a finite group of isometries of $\Sigma$, we say $u$ is $G$-invariant if $u$ is $\rho$-even for each $\rho\in G$.  It is convenient to define
\begin{align*}
\Ecali^G : = \{ u \in \Ecali : u \circ \rho = u, \rho \in G\}.
\end{align*} 
Given an isometry $\rho$ of $\Sigma$, we denote by $\Sigma^\rho$ the fixed point set of $\rho$.  We say an isometry $\rho$ of $\Sigma$ is a \emph{reflection} if for some $p \in \Sigma^\rho$, the tangent map $T_p \rho$ is a reflection in the Euclidean space $(T_p \Sigma, \left. g\right|_p)$ with respect to a hyperplane.  Finally, if $G$ is a finite group of isometries of $\Sigma$, we denote by $\Sigma/G$ the quotient and $\pi_G : \Sigma \rightarrow \Sigma/G$ the quotient map.  We will sometimes write $\pi_G(\Sigma)$ instead of $\Sigma/G$ for convenience.

\begin{lemma}
\label{L2nd}
Suppose that $\group$ is a finite group of isometries of $\Sigma$ and $u$ is a $\group$-invariant eigenfunction with exactly two nodal domains.  
Then the following hold, where $\Omega$ is either of the nodal domains of $u$:
\begin{enumerate}[label=\emph{(\roman*)}]
\item  $\pi_G (\Omega)$ is connected and $\pi_G (\Omega\cap \partial \Sigma) \neq \emptyset$.
\item  If $K$ is an index two subgroup of $G$, $\rho \in G\setminus K$ is a reflection, and $\pi_K(\Sigma^\rho)$ separates $\pi_K(\Sigma)$, then $\pi_G(\Omega\cap\Sigma^\rho) \neq \emptyset$.
\end{enumerate}
\end{lemma}
\begin{proof}
Because the Steklov eigenfunction $u$ is nonconstant, $\int_{\partial \Sigma} u =0$, hence $\Omega \cap \partial \Sigma \neq \emptyset$ since $u$ has exactly two nodal domains.  Applying the map $\pi_G$, item (i) follows.

Next consider the quotient $\pi_K(\Sigma)$.  Because $K$ has index two, $K$ is normal in $G$ and $\rho$ preserves each $K$ orbit.  Therefore, there is a well-defined involution $\rhotilde: \pi_K(\Sigma) \rightarrow \pi_K(\Sigma)$ satisfying $\rhotilde \circ \pi_K = \pi_K \circ \rho$.

Since $\rho$ is a reflection, since $\pi_K(\Sigma^\rho)$ separates $\pi_K(\Sigma)$, and since $\pi_K(\Sigma)$ is connected, $\pi_K(\Sigma)\setminus \pi_K(\Sigma^\rho)$ has precisely two components, which are moreover exchanged by $\rhotilde$.  Therefore, $\pi_K(\Omega)$ has nonempty intersection with each of these components, and it follows that $\pi_K(\Omega \cap \Sigma^\rho) \neq \emptyset$ since $\pi_K(\Sigma^\rho)$ separates $\pi_K(\Sigma)$ and $\pi_K(\Omega)$ is connected. 
\end{proof}

\subsection*{Surfaces and symmetry groups}
$\phantom{ab}$
\nopagebreak

We will be primarily interested in situations where $\Sigma$ is a surface and the quotient orbifold $\pi_G(\Sigma)$ is a polygon, and in the remainder of this section, we assume $\Sigma$ is two-dimensional.  The following definition will be convenient for our applications: 

\begin{definition}
\label{Dirred}
Consider a polygon whose cyclically ordered boundary edges $e_1, \dots, e_n$ are assigned labels $l_1, \dots, l_n$.  We say that the labeling is \emph{reducible} if there exist $1\leq i \leq j\leq n$ such that
\begin{align*}
\{ l_i, l_{i+1}, \dots, l_j\}= \bigcup_{i=1}^n \{l_i \} \quad \text{and} \quad
\{ l_j, \dots, l_n, l_1, \dots, l_i\}=\bigcup_{i=1}^n \{l_i\}. 
\end{align*}
We say the labeling is \emph{irreducible} if it is not reducible. 
\end{definition}

\begin{remark}
\label{Rirred}
If a labeled polygon with at least three edges has a pair of adjacent edges each of whose labels is assigned to no other edge, then clearly the labeling is irreducible.  This property holds for most of the labeled polygons arising from the surfaces $\Sigma$ considered in Section \ref{Sapplications} (compare with Remark \ref{Rred}).
\end{remark}

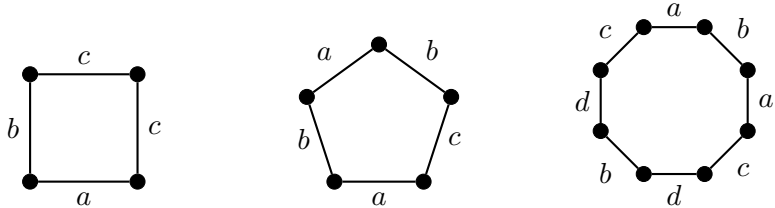
\begin{figure}[h]
\centering
\begin{subfigure}[t]{.3\textwidth}
\centering
\begin{tikzpicture}[mystyle/.style={draw,shape=circle, minimum size=.1cm, inner sep=2pt, fill=black}]
\node[regular polygon,regular polygon sides=4,minimum size=2cm] (p) {}; 
\foreach\x in {1,...,4}{\node[mystyle] (p\x) at (p.corner \x){};} 

\draw[thick] (p3)--(p4) node[midway,  below]{$a$};
\draw[thick] (p4)--(p1) node[midway,  right]{$c$};
\draw[thick] (p1)--(p2) node[midway, above]{$c$};
\draw[thick] (p2)--(p3) node[midway,  left]{$b$};

\end{tikzpicture}
\end{subfigure}
~
\begin{subfigure}[t]{.3\textwidth}
\centering
\begin{tikzpicture}[mystyle/.style={draw,shape=circle, minimum size=.1cm, inner sep=2pt, fill=black}]
\node[regular polygon,regular polygon sides=5,minimum size=2cm] (p) {}; 
\foreach\x in {1,...,5}{\node[mystyle] (p\x) at (p.corner \x){};} 

\draw[thick] (p3)--(p4) node[midway, below]{$a$}; 
\draw[thick] (p4)--(p5) node[midway, right]{$c$};
\draw[thick] (p5)--(p1) node[midway, above right]{$b$};
\draw[thick] (p1)--(p2) node[midway, above left ]{$a$};
\draw[thick] (p2)--(p3) node[midway,  left]{$b$};
\end{tikzpicture}
\end{subfigure}
~
\begin{subfigure}[t]{.3\textwidth}
\centering
\begin{tikzpicture}[mystyle/.style={draw,shape=circle, minimum size=.1cm, inner sep=2pt, fill=black}]
\node[regular polygon,regular polygon sides=8,minimum size=2.1cm] (p) {}; 
\foreach\x in {1,...,8}{\node[mystyle] (p\x) at (p.corner \x){};} 
\draw[thick] (p5)--(p6) node[midway, below]{$d$};
\draw[thick] (p6)--(p7) node[midway, below right]{$c$};
\draw[thick] (p7)--(p8) node[midway, right]{$a$};
\draw[thick] (p8)--(p1) node[midway, above right]{$b$};
\draw[thick] (p1)--(p2) node[midway, above ]{$a$};
\draw[thick] (p2)--(p3) node[midway,  above left]{$c$};
\draw[thick] (p3)--(p4) node[midway,  left]{$d$}; 
\draw[thick] (p4)--(p5) node[midway, below left ]{$b$};

\end{tikzpicture}
\end{subfigure}
\caption{A square with an irreducible labeling; reducible labelings on a pentagon and octagon.}
\label{F1}
\end{figure}

It is a standard fact \cite{Cheng, FraserSchoen, Kokarev} that the nodal set $\Ncal_u$ of a Steklov eigenfunction on a surface is a piecewise-$C^1$ graph consisting of finitely many edges where an even number of edges meet each vertex at equal angles.  We call a $C^1$ arc in $\Ncal_u$ whose endpoints lie on $\partial \Sigma$ a \emph{nodal line}.  

\begin{prop}
\label{Psym}
Suppose $G$ is a finite group of isometries of $\Sigma$ generated by reflections $\{\rho_i\}_{i \in I}$  and that the following hold:
\begin{enumerate}[label=\emph{(\alph*)}]
\item $\pi_G(\Sigma)$ is a polygon with an irreducible labeling, where each edge is contained in one of the sets $\pi_G(\partial \Sigma)$ or $\pi_G(\Sigma^{\rho_i})$ for some $i\in I$ and is assigned the label of its containing set. 
\item For each $i\in I$, there is an index two subgroup $K_i$ of $G$ such that $\rho_i \notin K_i$ and $\pi_{K_i}(\Sigma^{\rho_i})$ separates $\pi_{K_i}(\Sigma)$. 
\end{enumerate}
Then there is no $G$-invariant eigenfunction with exactly two nodal domains. 
\end{prop}
\begin{proof}
Let $u$ be a nonconstant $G$-invariant Steklov eigenfunction.  We may suppose that $\pi_G(\Ncal_u)$ is a single arc intersecting the boundary of $\pi_G(\Sigma)$ in exactly two places, because otherwise---since $\pi_G(\Sigma)$ is a polygon and is in particular simply connected---at least one of $\pi_G(\{u> 0\})$ and $\pi_G(\{u<0\})$ fails to be connected, so $u$ has more than two nodal domains by \ref{L2nd}(i).  

Since the labeling of $\pi_G(\Sigma)$ is irreducible, some component of $\pi_G(\Sigma \setminus \Ncal_u)$ fails to intersect one of the $\pi_G( \Sigma^{\rho_i})$, so $u$ has more than two nodal domains by \ref{L2nd}(ii).
\end{proof}

\subsection*{Antipodal isometries}
\phantom{ab}
\nopagebreak

We now study the case where $\Sigma$ is a surface admitting an orientation reversing, fixed-point free, involutive isometry.  

\begin{lemma}
\label{Lgen0}
Suppose $\Sigma$ has genus zero and admits an orientation reversing, fixed-point free, involutive isometry $\alpha$.  Then $\firsteigen$ is $\alpha$-odd. 
\end{lemma}

\begin{proof}
Consider a nonzero $u\in \firsteigen$ with its two nodal domains.  Without loss of generality it suffices to assume $u$ is either $\alpha$-even or $\alpha$-odd. Note that $u$ is $\alpha$-even (respectively, $\alpha$-odd) if and only if $\alpha$ preserves (exchanges)  the nodal domains of $u$.

Let $G$ be the group of order two generated by $\alpha$.  Since $\alpha$ is fixed-point free, $\pi_G : \Sigma \rightarrow \pi_G (\Sigma)$ is a covering map, so a path joining any $p\in \Sigma$ to $\alpha(p)$ projects to an essential loop in $\pi_G(\Sigma)$.

  Because $\alpha$ is orientation reversing  and $\Sigma$ has genus zero, we may identify $\pi_G(\Sigma)$ with a subset of $\RP^2$.   But any two essential loops in $\RP^2$ must meet, so $\alpha$ must exchange the nodal domains, and hence $u$ is $\alpha$-odd. 
\end{proof}

\begin{remark}
If $\Sigma$ is as in Lemma \ref{Lgen0} but has positive genus, it is no longer the case that any two essential loops in $\pi_G(\Sigma)$ must meet.  For example, there exists a rectangular torus with four disks removed $\Sigma$ and a glide reflection $\alpha$ such that $\Sigma$ is a union of the closures of two disjoint $\alpha$-invariant domains.  In this example, $\pi_G(\Sigma)$ is a Klein bottle with two disks removed.
\end{remark}

\begin{prop}
\label{Pann}
Suppose $\Sigma$ is an annulus admitting an orientation reversing, fixed-point free, involutive isometry $\alpha$.  Moreover, suppose there is a three-dimensional space $\Vs$ of $\alpha$-odd eigenfunctions for which each nonzero element has exactly two nodal domains.  Then $\firsteigen = \Vs$.
\end{prop}
\begin{proof}
In order to prove that $\firsteigen = \Vs$, it suffices to show  for each nonzero $u \in \firsteigen$ that  $\int_{\partial \Sigma} u \varphi \neq 0$ for some $\varphi \in \Vs$.  

First note that $\alpha$ exchanges the components of $\partial \Sigma$ since $\alpha$ is orientation reversing and fixed-point free, and that $\firsteigen$ is $\alpha$-odd by Lemma \ref{Lgen0}. 

Consider the two alternatives for a nonzero  $u\in \firsteigen$.

Case 1: $u$ does not change sign on either component of $\partial \Sigma$.   Then $u$ has opposite signs on the two boundary components.  Choose a point $p\in \partial \Sigma$ and consider the linear map $T: \Vs\rightarrow \R^2$ defined by $T \varphi = ( \varphi(p), \frac{d\varphi}{ds}(p))$, where $\frac{d\varphi}{ds}(p)$ is the derivative with respect to arclength along $\partial \Sigma$ at $p$.  Since $\dim \Vs = 3$, there exists a nonzero $\varphi \in \Vs$ satisfying $\varphi(p) = \frac{\partial \varphi}{\partial s}(p) = 0$.  Since $\varphi$ has exactly two nodal domains and $\Ncal_\varphi = \alpha(\Ncal_\varphi)$, $\Ncal_\varphi$ consists of a topological circle joining $p$ and $\alpha(p)$  dividing $\Sigma$ into two connected sets. It follows that $\varphi$ cannot change sign on either component of $\partial \Sigma$.

Case 2: $u$ changes sign on a component of $\partial \Sigma$. Then there exist $p$ and $q$ in this component of $\partial \Sigma$ such that $\left. u \right|_{\partial \Sigma}$ changes sign at $p$ and at $q$.  Since $u$ has only two nodal domains, $\Sigma$ is an annulus, and $\Ncal_u = \alpha(\Ncal_u)$, it follows that $\Ncal_u = \ell \cup \alpha(\ell)$, where $\ell$ is a nodal line joining $p$ to $\alpha(q)$.  Consider the linear map $T: \Vs \rightarrow \R^2$ defined by $T \varphi = ( \varphi(p), \varphi(q))$.  Since $\dim \Vs = 3$, there exists a nonzero $\varphi \in \Vs$ satisfying $\varphi(p) = \varphi(q) = 0$. Since $\varphi$ has exactly two nodal domains, $\Sigma$ is an annulus, and $\Ncal_\varphi = \alpha(\Ncal_\varphi)$,  then $\Ncal_\varphi = \ell' \cup \alpha(\ell')$, where $\ell'$ is a nodal line joining $p$ to $\alpha(q)$.

In either case, Lemma \ref{Lsym} applies with $u$ and $\varphi$, so $\int_{\partial \Sigma} u \varphi \neq 0$. 
\end{proof}

\begin{remark}
If $\Sigma$ is an annulus (or any genus zero surface) and $\Vs$ is a space of eigenfunctions for which each nonzero element has exactly two nodal domains, arguments in \cite{FraserSchoen,Kokarev} show that $\dim \Vs \leq 3$. 
\end{remark}

\section{Free Boundary Minimal Submanifolds in $\B^n$ and Symmetries}
Consider a \emph{free boundary minimal submanifold} of the ball, that is, a properly immersed minimal submanifold $\Sigma^m$  of $\B^n\subset \R^n$ such that $\Sigma$ meets $\partial \B^n$ orthogonally along $\partial \Sigma$. We denote by $\Coord = \Coord(\Sigma)$ the span of the coordinate functions on $\Sigma$, which are Steklov eigenfunctions with eigenvalue $1$.

One of the purposes of this paper is to prove that $\firsteigen = \Coord$ for certain free boundary minimal surfaces; we remark this condition forces isometric rigidity:
\begin{prop}
\label{Piso}
Let $\Sigma^m$ be a Riemannian manifold admitting a full proper isometric immersion $X: \Sigma^m \rightarrow \B^{m+1}$ as a free boundary minimal hypersurface, and suppose $\Ecal_1 =\Coord$.   Then any other isometric immersion $Y: \Sigma\rightarrow \B^n$ of $\Sigma$ as a free boundary minimal submanifold is congruent to $X$, in other words $Y = T\circ X$ for some linear isometric embedding $T: \R^{m+1}\rightarrow \R^n$. 
\end{prop}
\begin{proof}
The assumptions imply the components of $Y$ are linear combinations of the components of $X$, so the resulting linear map $T: \R^{m+1} \rightarrow \R^n$ such that $Y = T\circ X$ has rank $m+1$.  Since $X(\partial \Sigma)\subset \partial \B^{m+1}$ and  $Y(\partial \Sigma) = (T\circ X)(\partial \Sigma) \subset \partial \B^n$, it follows that $T$ is an isometry onto its image. 
\end{proof}

We now study free boundary minimal submanifolds of $\B^n$ which are preserved by a reflection of the Euclidean space $\R^n$ whose fixed-point set is a given vector subspace $V\subset \R^n$.  Denote by $V^\perp$ the orthogonal complement of $V$ in $\R^n$, and define the reflection $\Rcapunder_V : \R^n \rightarrow \R^n$ with respect to $V$, by
\begin{align*}
\Rcapunder_V : = \Pi_V - \Pi_{V^\perp},
\end{align*}
where $\Pi_V$ and $\Pi_{V^\perp}$ are the orthogonal projections of $\R^n$ onto $V$ and $V^\perp$ respectively.  Note in particular that $\Rcapunder_{\{ 0\}}$ is the antipodal map, where $\{ 0\}$ is the trivial subspace of $\R^n$.  Finally, we denote
\begin{align*}
\Coord_V = \Coord_V(\Sigma) : = \{ \varphi \in \Coord : \left. \varphi \right|_{V^\perp} = 0\}.
\end{align*}

\begin{cor}
\label{c}
For each hyperplane $P$ of reflection symmetry of $\Sigma$, we have either $\Acal_{\Rcap_P}(\firsteigen) = \{ 0\}$ or $\Acal_{\Rcap_P}(\firsteigen) = \Coord_{P^\perp}$. 
\end{cor}
\begin{proof}
Choose a nonzero $\varphi \in \Coord_{P^\perp}$.  Since $P$ separates $\Sigma$, for any nonzero $u\in \Acal_{\Rcap_P}(\firsteigen)$, $u$ and $\varphi$ satisfy the conditions of \ref{Lsym} (with $\rho = \Rcap_P$).
\end{proof}

\section{On Free Boundary Minimal Annuli Embedded in the Ball}

An important property of the coordinate functions for free boundary minimal surfaces in $\B^3$ is the following result of Lima-Menezes \cite{LimaMenezes}:

\begin{theorem}[Two-piece property {\cite[Theorem A]{LimaMenezes}}]
Every nontrivial linear function $\varphi \in \Coord$ on an embedded free boundary minimal surface
 $\Sigma \subset \B^3$ has exactly two nodal domains. 
  \end{theorem}
\noindent

We are grateful to W. Kusner for pointing out \cite{WKusner} how (i) implies (ii) in the following extension of \cite[Prop. 8.1]{FraserSchoen}: 
\begin{corollary}[Radial graphs]
\label{Cfs}
Let $\Sigma \subset \B^3$ be an embedded free boundary minimal surface of genus zero and at least two boundary components.  Then: 
\begin{enumerate}[label=\emph{(\roman*)}]
\item $\Sigma$ is a radial graph, in the sense that $0\notin \Sigma$ and each ray from $0$ transversally intersects $\Sigma$ at most once; 
\item each component of $\partial \Sigma$ is a convex curve in $\Sph^2$.
\end{enumerate}
\end{corollary}
\begin{proof}
Item (i) follows as in \cite[Prop. 8.1]{FraserSchoen} by replacing the condition that $\sigma_1 = 1$ with the two-piece property.  Next, observe that (i) and the two-piece property imply each great circle separates the radial projection of $\Sigma$ to~$\Sph^2$ into precisely two components.  Hence each great circle has at most two transverse intersections with each component of $\partial \Sigma$, which yields (ii).
\end{proof}

We are now ready to prove Theorem \ref{Tcc}. 

\begin{proof}[Proof of Theorem \ref{Tcc}]
By Corollary \ref{Cfs}(i), the radial projection from $A$ to $\Sph^2$ is an antipodally-equivariant diffeomorphism onto its image.  Thus the restriction $\alpha$ of $\Rcap_{\{0\}}$  to $A$ is an orientation reversing, fixed-point free,  involutive isometry.

 Because each nonzero $\varphi \in \Coord(A)$ has exactly two nodal domains by the two-piece property, we may apply Proposition \ref{Pann} with $\Vs = \Coord(A)$ to conclude that $\firsteigen(A) = \Coord(A)$.  A result of Fraser-Schoen \cite[Theorem 6.6]{FraserSchoen} now implies that $A$ is the critical catenoid. 
\end{proof}

\section{Main Results}
\label{Sthms}

Under certain symmetry assumptions on a free boundary minimal surface $\Sigma\subset \B^3$, we now show that $\firsteigen(\Sigma) = \Coord(\Sigma)$ and in particular that $\sigma_1(\Sigma) = 1$.

\subsection*{Symmetry groups}
$\phantom{ab}$
\nopagebreak

Any finite isometry group of $\Sph^2 = \partial \B^3$ generated by reflections through planes has a geodesic polygon fundamental domain,  each of whose vertex angles is an integer fraction of $\pi$.
These groups are
\begin{align*}
*mm \quad
\ast \!22m \quad
\ast \!233 \quad
\ast \!234 \quad
\ast \!235, \quad
m \in \N
\end{align*}
in Conway-Thurston notation \cite{Conway}, where the numbers after the star denote the integer fractions. For example, $*mm$ has a fundamental digon with vertex angles $\pi/m$, and $*klm$ has a fundamental trigon whose vertex angles are $\pi/k, \pi/l$, and $\pi/m$.

For each $m \geq 2$, we will consider also the group $2*m$, which includes $*mm$ as an index two subgroup and contains also reflections about lines (half-turns) between adjacent mirror planes of $*mm$. 

\subsection*{Results for free boundary minimal surfaces in  $\B^3$}
$\phantom{ab}$
\nopagebreak

Let $\Sigma \subset \B^3$ be an embedded $G$-invariant free boundary minimal surface  satisfying the following (recall the notation of \ref{Psym}): 
\begin{assumption}
\phantom{ab}
\label{Ag}
\begin{enumerate}[label={(\alph*)}]
\item $G$ is either a finite group generated by reflections in hyperplanes, or $\group = 2\!*\!m$ for some $m\geq 2$.  In the latter case, we suppose further that the projection of the halfturn lines to $\pi_H(\Sigma)$ separates $\pi_{H}(\Sigma)$ and that $\pi_H(\Sigma)$ is simply connected, where $H = *mm$.
\item $\pi_G(\Sigma)$ is a polygon with an irreducible labeling, where each edge is contained in one of the sets $\pi_G(\partial \Sigma)$ or $\pi_G(\Sigma^{\rho_i})$ for some $i\in I$ and is assigned the label of its containing set. 
 \end{enumerate}
 \end{assumption}

\begin{theorem}
\label{Tgroup}
Suppose $\Sigma \subset \B^3$ is an embedded free boundary minimal surface satisfying Assumption \ref{Ag}.  Then $\firsteigen = \Coord$ and $\sigma_1(\Sigma)=1$. 
 \end{theorem}
\begin{proof}
Note that the hypotheses of Proposition  \ref{Psym} apply, so $\firsteigen^G = \{ 0\}$. 

We first consider the case where $G$ is generated by reflections through planes.  For each plane $P$ of symmetry of $\Sigma$, Corollary \ref{c} implies $\Acal_{\Rcap_P}(\firsteigen)=\{ 0\}$ or $\Acal_{\Rcap_P}(\firsteigen) = \Coord_{P^\perp}$.  The latter alternative holds for some such $P$ since $\firsteigen^G$ is trivial, so $\Coord \subset \firsteigen$ and hence the latter alternative holds for every such $P$.  The orthogonal complement of $\Coord$ in $\firsteigen$ is then contained in $\firsteigen^G = \{ 0\}$, so $\firsteigen = \Coord$.

We next consider the case where $G = 2*m$.  Let $H = *mm$ be the index two subgroup of $G$ generated by reflections through planes and pick a halfturn $\Rcap_L \in G \setminus H$ about a line $L$.  The first eigenspaces $\firsteigen$ splits as an orthogonal sum of subspaces
\begin{align*}
\firsteigen = \firsteigen^H \oplus (\firsteigen^H)^\perp &= \Scal_{\Rcap_L} ( \firsteigen^H) \oplus \Acal_{\Rcap_L}(\firsteigen^H) \oplus (\firsteigen^H)^\perp \\
&= \firsteigen^G \oplus \Vcal_1 \oplus \Vcal_2 =  \Vcal_1 \oplus \Vcal_2, 
\end{align*}
where $\Vcal_1 = \Acal_{\Rcap_L}(\firsteigen^H)$ and $\Vcal_2 =  (\firsteigen^H)^\perp $.

The action of $H$ on $\R^3$ fixes pointwise a one-dimensional subspace $W\subset \R^3$ (the intersection of the planes corresponding to reflections generating $H$) and fixes as a set its orthogonal complement $W^\perp$ in $\R^3$.  We will argue below that $\Vcal_1$ is either trivial or equal to $\Coord_{W}$, and that $\Vcal_2$ is either trivial or equal to $\Coord_{W^\perp}$.  This will complete the proof, since the nontriviality of $\firsteigen$ implies the nontriviality of one of $\Vcal_1$ and $\Vcal_2$, so that $\firsteigen \supset \Coord = \Coord_W \oplus \Coord_{W^\perp}$ and therefore $\firsteigen =  \Coord_W \oplus \Coord_{W^\perp}$.

The claimed dichotomy for $\Vcal_2$ follows by arguing as in the first case above.
Now consider a nonzero $u \in \Vcal_1$.  Since $u$ is $H$-invariant, $H$ and $u$ satisfy the hypotheses of Lemma \ref{L2nd}, with $H$ taking the role of $G$.  Similarly for any nonzero $\varphi \in \Coord_W$, by the two piece property \cite{LimaMenezes}, $H$ (again in the role of $G$) and $\varphi$ also satisfy the hypotheses of Lemma \ref{L2nd}. 

Since $\pi_{H}(\Sigma \cap L)$ separates $\pi_{H}(\Sigma)$  it follows from Lemma \ref{L2nd}(i) that $\pi_{H}(\Ncal_f) = \pi_{H}(L)$, where $f$ is either $\varphi$ or $u$.  By Lemma \ref{Lsym}, $\int_{\partial \Sigma} u \varphi \neq 0$.  It follows that $\Vcal_1 = \Coord_W$. 
\end{proof}

\subsection*{Results for minimal surfaces in $\Sph^3$}
$\phantom{ab}$
\nopagebreak

Most of the framework developed earlier leads to analogous conclusions for the first Laplace eigenspace $\Ecal_{\lambda_1}$ on a closed, $G$-invariant minimal surface $\Sigma$ embedded in $\Sph^3$ satisfying the following:

\begin{assumption}
\label{Agsph}
\phantom{ab}
\begin{enumerate}[label={(\alph*)}]
\item $G$ is a finite group generated by reflections in hyperplanes of $\R^4$.
\item $\pi_G(\Sigma)$ is a polygon with an irreducible labeling, where each edge is contained in one of the sets $\pi_G(\Sigma^{\rho_i})$ for some $i\in I$ and is assigned the label of its containing set. 
\end{enumerate}
\end{assumption}

\begin{theorem}
\label{Tsphere}
Suppose $\Sigma \subset \Sph^3$ is an embedded, closed minimal surface satisfying \ref{Agsph}.  Then $\Ecal_{\lambda_1} = \Coord$ and $\lambda_1(\Sigma) = 2$.
\end{theorem}
\noindent We omit the proof, which is similar to the proof of Theorem \ref{Tgroup}.

\section{Applications}
\label{Sapplications}
We now describe various applications of Theorem \ref{Tgroup} to show that $\firsteigen = \Coord$ for known or expected examples of free boundary minimal surfaces in $\B^3$, and of Theorem \ref{Tsphere} to show that $\Ecal_{\lambda_1} = \Coord$ on certain embedded minimal surfaces in $\Sph^3$. 

\subsection*{Stackings of the disk}
$\phantom{ab}$
\nopagebreak

 Doublings and triplings of the equatorial disk via gluing methods have been constructed in \cite{Zolotareva} and \cite{KapWiygul}.  More generally, Kapouleas-Wiygul observe  \cite[Remark 1.2]{KapWiygul} that their tripling construction can be extended to produce \emph{stackings} of the disk, that is, free boundary minimal surfaces resembling $N$ parallel copies of the equatorial disk, with each adjacent pair of copies joined by $m$ catenoidal strips in maximally symmetric fashion.  We show that for each $N\geq 2$, any surface $\Sigma$ consistent with their description satisfies the hypotheses of Theorem \ref{Tgroup}, hence $\firsteigen(\Sigma) = \Coord(\Sigma)$.   These examples fall into two classes depending on the parity of the number of copies $N$.

When $N=2k+1$ is odd, the symmetry group is $G = 2*m$ and the expected surface has connected boundary and genus $\frac{1}{2}(m-1)(N-1)$.  In particular, $\pi_G(\Sigma)$ is homeomorphic to a fundamental domain $D$ for $\Sigma$ bounded by mirror planes  $P_1$ and $P_2$, a halfturn line $L$, and $\partial \B^3$.  Its boundary is a $(3+k)$-gon consisting of edges in the following order: an edge along $\partial \Sigma$, an edge along $L$, edges alternating between $P_1$ and $P_2$ corresponding to the catenoidal strips joining adjacent disks, and an edge in $P_1$ joining the center of the bottom disk back to $\partial \Sigma$.

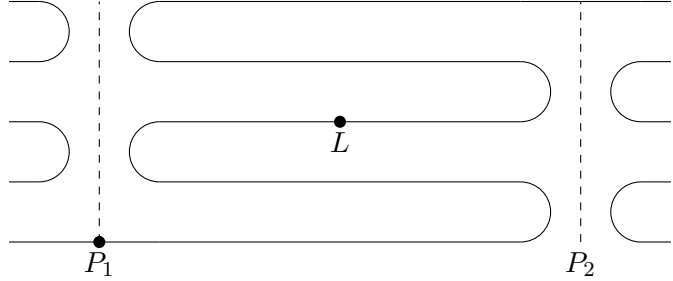
\begin{figure}[h]
\begin{tikzpicture}[scale=.8]
\foreach \i in {-2, -1, 0, 1, 2}{\draw[] (-3, \i)--(3, \i);} 
\foreach \i in {-1, 1}{\draw[] (3*\i, 2*\i)--(4*\i , 2*\i);} 
\foreach \i in {-1, 1}{\draw[dashed] (-4*\i , -2)--(-4*\i, 2);} 
\foreach \i in {-2, 0} {\draw[] (3, \i) arc (-90:90:.5);} 
\foreach \i in {-1, 1}{\draw[] (-3, 1*\i ) arc (270:90:.5);} 
\draw[] (0,0) node[below]{$L$};
\fill[black] (0, 0) circle (.1); 
\fill[black] (-4, -2) circle (.1);
\draw[] (-4,-2) node[below]{$P_1$};
\draw[] (4,-2) node[below]{$P_2$};
\foreach \i in {-2, -1, 0, 1, 2}{\draw[] (-5, \i)--(-5.5, \i);} 
\foreach \i in {-2, -1, 0, 1, 2}{\draw[] (5, \i)--(5.5, \i);} 
\foreach \i in {-1, 1}{\draw[] (5*\i, 2*\i)--(4*\i , 2*\i);} 
\foreach \i in {-1, 1} {\draw[] (-5, \i) arc (-90:90:.5);} 
\foreach \i in {-2, 0}{\draw[] (5, 1*\i ) arc (270:90:.5);} 
\end{tikzpicture} 
\caption{Side view schematic for a stacking of $5$ equatorial disks.  $D$ is bounded by an edge in $\partial \Sigma$ joining the marked points on $P_1$ and $L$, an edge in $L$ joining the marked point to the center of the middle disk, edges in each $P_1$ and $P_2$ joining neighboring disks via catenoidal strips, and an edge in $P_1$ joining the center of the bottom disk and the marked point.}
\end{figure}

If $N=2k$ is even, the symmetry group is $G = *22m$ and the expected surface $\Sigma$ has $m$ boundary components and genus $\frac{1}{2}(m-1)(N-2)$. Then, $\pi_G(\Sigma)$ is homeomorphic to a fundamental domain $D$ for $\Sigma$ bounded by mirror planes $P_1, P_2, P_3$, and $\partial \B^3$.  Its boundary is a $(3+k)$-gon consisting of the following: an edge along $\partial \Sigma$, an edge in $P_3$ along a quarter-catenoidal waist in the catenoidal strip joining the middle disks, edges alternating between $P_2$ and $P_1$ corresponding to the catenoidal strips joining adjacent disks, and a final edge in $P_1$ joining the center of the bottom disk back to $\partial \Sigma$.
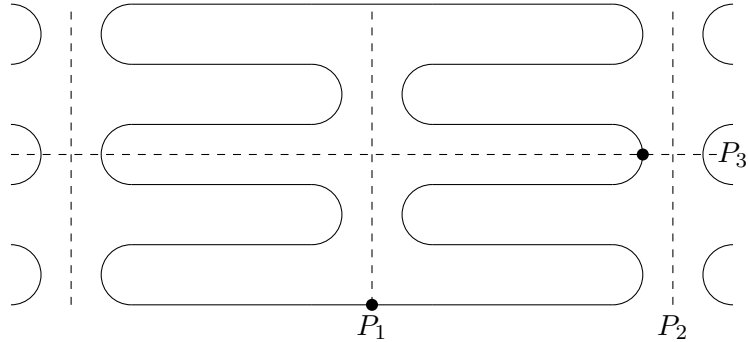
\begin{figure}
\centering
\begin{tikzpicture}[scale=.8]

\foreach \i in {-2, -1, 0, 1, 2, 3}
	{
	\draw[] (-4, -.5+\i)--(-1, -.5+\i);
	\draw[] (1, -.5+\i)--(4, -.5+\i);
	}
\draw[dashed] (0, -2.5)--(0, 2.5); 
\draw[dashed] (5, -2.5)--(5, 2.5); 
\draw[dashed] (-5, -2.5)--(-5, 2.5); 
\draw[dashed] (-6, 0)--(5.75, 0);

\foreach \i in {-1,0, 1}{\draw[] (-4, -.5+2*\i ) arc (270:90:.5);} 
\foreach \i in {-1,0, 1}{\draw[] (4, -.5+2*\i ) arc (-90:90:.5);} 
\foreach \i in {0, 1}{\draw[] (1, -1.5+2*\i ) arc (270:90:.5);} 
\foreach \i in {0, 1}{\draw[] (-1, -1.5+2*\i ) arc (-90:90:.5);} 
\foreach \i in {-1, 1}{\draw[] (-1,2.5*\i)--(1, 2.5*\i);}

\foreach \i in {-1,0, 1}{\draw[] (6, -.5+2*\i ) arc (270:90:.5);} 
\foreach \i in {-1,0, 1}{\draw[] (-6, -.5+2*\i ) arc (-90:90:.5);} 

\draw[] (0,-2.5) node[below]{$P_1$};
\draw[] (5,-2.5) node[below]{$P_2$}; 
\draw[] (6,0) node[]{$P_3$};
\draw[] (-6,0) node[]{\phantom{$P_3$}};

\fill[black] (0, -2.5) circle (.1); 
\fill[black] (4.5, 0) circle (.1); 
\end{tikzpicture} 
\caption{Side view schematic for a stacking of $6$ equatorial disks.  The six edges of $D$ consist of an edge along $\partial \Sigma$ joining the marked points on $P_1$ and $P_3$, an edge in $P_3$, three edges joining neighboring disks via catenoidal strips (two in $P_2$ and one in $P_1$), and an edge in $P_1$ joining the center of the bottom disk to the bottom marked point.}
\end{figure}

In each of the preceding cases, $\Sigma$ satisfies Assumption \ref{Ag} and the labeling is irreducible by Remark \ref{Rirred}, so Theorem \ref{Tgroup} applies.  When $N=3$, that the triplings in  \cite[Theorem 7.40]{KapWiygul} are actually $G = 2*m$-invariant for some $m\geq 2$ follows from \cite[Section 3]{KapWiygul}.  When $N=2$, it was shown in \cite{McGrath} that the doublings in \cite{Zolotareva} have $\sigma_1=1$. 

\subsection*{Desingularizations of disks meeting along a diameter}
$\phantom{ab}$
\nopagebreak

In the introduction of \cite{KapLi},  Kapouleas-Li discuss work in preparation on constructions of free boundary minimal surfaces by desingularizing two disks which meet orthogonally along a diameter of $\B^3$.  We show here that any free boundary minimal surfaces with certain symmetries consistent with desingularizations of $m\geq 2$ disks that meet equiangularly along a diameter are covered by Theorem \ref{Tgroup} and would necessarily have $\firsteigen = \Coord$.

Observe first that the configuration consisting of $m$ disks meeting at equal angles along a diameter is invariant under either $2*m$ or $*22m$.  The examples fall into two classes depending on the parity of the number $N$ of Scherk-like handles used.  In either class, however, the symmetry group is independent of $N$. 

\begin{figure}[h]
\centering
\def\ra{4} 
\begin{subfigure}[]{.4\textwidth}
\centering
\begin{tikzpicture}[scale=.7] 
\draw [draw=black, fill=light-gray] (0, 0) ellipse ({\ra/sqrt(2)} and {\ra});
\draw[dashed] (0, -\ra)--(0, \ra); 
\draw[dashed] (-{\ra/sqrt(2)}, 0)--({\ra/sqrt(2)}, 0); 

\draw[] (0,-\ra) node[below]{$P_1$}; 
\draw[] ({\ra/sqrt(2)}, 0) node[right]{$P_3$};
\draw[] (-{\ra/sqrt(2)}, 0) node[left]{\phantom{$P_3$}};

\foreach \j in {0, 1, 2}
{
\foreach \i in {-1, 1}
{
\draw [draw=black, fill=white] (0, {\i*(\j*\ra/3+\ra/6)}) circle ({\ra/12});
}
}
\end{tikzpicture} 
\end{subfigure}
\begin{subfigure}[]{.4\textwidth}
\centering
\begin{tikzpicture}[scale=.7]
\draw[] (0,-\ra) node[below]{$P_1$}; 
\draw[] ({\ra/sqrt(2)}, 0) node[right]{$L$};

\draw[] (-{\ra/sqrt(2)}, 0) node[left]{\phantom{$L$}};
   \draw [fill=light-gray, domain=-83:180+83, samples=100] plot ({\ra/sqrt(2)*cos(\x)}, {\ra*sin(\x)}); 
	
\draw[dashed, black] (0, -\ra+.1)--(0, \ra); 
\draw[black] (0, 0)--({\ra/sqrt(2)}, 0); 
\draw[dotted, black] (0, 0)--({-\ra/sqrt(2)}, 0); 

   \draw [fill=white, domain=3:177, samples=100] plot ({\ra/11*cos(\x)}, {-4+\ra/11*sin(\x)});

\foreach \j in {0, 1, 2}
{
\draw [fill=white](0, {\j*4*\ra/11+\ra/11}) circle ({\ra/11});
}
\foreach \j in {0, 1}
{
\draw [fill=white] (0, {-\j*4*\ra/11-3*\ra/11}) circle ({\ra/11});
}

\end{tikzpicture}
\end{subfigure}
\caption{Schematics, viewed in a mirror plane $P_2$, for desingularizations of $m=2$ two disks.  The left example has six handles, two boundary components, and is $*22m$-invariant, while the right has five handles, one boundary component, and is $2*m$-invariant.}
\end{figure}
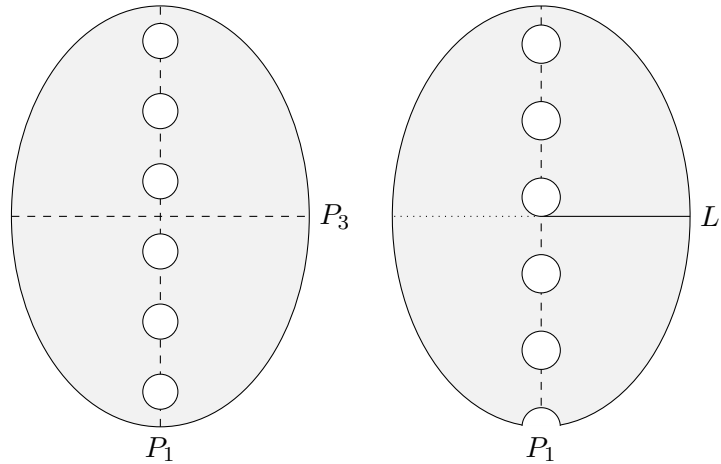

When $N=2k$ is even, the symmetry group is $G=*22m$ and the expected surface has two boundary components.  The orbifold $\pi_G(\Sigma)$ is homeomorphic to a fundamental domain $D$ on $\Sigma$ bounded by three mirror planes $P_1, P_2, P_3$ and $\partial \B^3$.  Its boundary is a $(3+2k)$-gon consisting of the following ordered edges: one along $\partial \Sigma$, one along $P_3$, and edges alternating between $P_1$ and $P_2$ traversing up the Scherk-like handles. 

When $N=2k+1$ is odd, the symmetry group is $G = 2*m$ and the expected surface has one boundary component.  The orbifold $\pi_G(\Sigma)$ is homeomorphic to a fundamental domain $D$ on $\Sigma$ bounded by mirror planes $P_1, P_2$, a halfturn line $L$, and $\partial \B^3$.  Its boundary is a $(4+2k)$-gon consisting of ordered edges, one in $\partial \Sigma$, one in $L$, and edges alternating between $P_2$ and $P_1$ traversing the Scherk-like handles. 

In each of the preceding cases, $\Sigma$ satisfies Assumption \ref{Ag} and the labeling is irreducible by Remark \ref{Rirred}, so Theorem \ref{Tgroup} applies. 

\subsection*{Genus zero examples with dihedral or Platonic symmetry}
$\phantom{ab}$

We have already observed that $\firsteigen = \Coord$ on a family of genus zero doublings of the disk \cite{Zolotareva} with dihedral $*22n$ symmetry and $n$ boundary components centered around a configuration of $n$ points equidistributed on an equatorial circle of $\partial \B^3$.  Here we catalog (see Table \ref{Tab1}) an additional infinite family of dihedrally-symmetric configurations of points, as well as 16 distinct configurations of points with Platonic symmetries, each of which  plausibly corresponds to a genus zero example to which Theorem \ref{Tgroup} applies.  The boundary components of a particular example are centered at the points of the corresponding configuration. 

\begin{table}[h]
\caption{Configurations with dihedral or Platonic symmetries.  By enforcing additional available symmetries, configurations $[2'33]$, $[23'3']$, and $[2'3'3']$ coincide with $[234']$, $[23'4]$, and $[2'3'4]$, respectively, leaving 16 Platonic configurations.}
\begin{tabular}{@{}l|cc@{}}
\toprule
Group                   & Configuration     & Ends             \\ \midrule
\multirow{2}{*}{$*22n$} & $[2'2n]=[22'n]$ & $n$ \\
				& $[2'2n']=[22'n']$ & $n+2$ \\ \midrule
\multirow{5}{*}{$*233$} & $[23'3]=[233']$   & $4$                       \\
                        & $[2'33]$          & $6$    \\
                        & $[23'3']$         & $8$    \\
                        & $[2'3'3]=[2'33']$ & $10$                    \\
                        & $[2'3'3']$        & $20$  \\ \midrule
\multirow{7}{*}{$*234$} & $[234']\cong[2'33]$          & $6$    \\
                        & $[23'4]\cong[23'3']$          & $8$  \\
                        & $[2'34]$          & $12$                     \\
                        & $[23'4']$         & $14$                      \\
                        & $[2'34']$         & $18$                      \\
                        & $[2'3'4]\cong[2'3'3']$         & $20$  \\
                        & $[2'3'4']$        & $26$                    \\ \midrule
\multirow{7}{*}{$*235$} & $[235']$          & $12$                   \\
                        & $[23'5]$          & $20$                      \\
                        & $[2'35]$          & $30$                      \\
                        & $[23'5']$         & $32$                      \\
                        & $[2'35']$         & $42$                      \\
                        & $[2'3'5]$         & $50$                      \\
                        & $[2'3'5']$        & $62$                      \\ \bottomrule
\end{tabular}
\label{Tab1}
\end{table}

 For each of the Platonic symmetry groups $*23n$ with $n=3,4,5$ (respectively tetrahedral, octahedral, and icosahedral) there are $2^3-1 =7$ configurations, corresponding to potential genus zero free boundary minimal surfaces through the (independent) choices of whether boundary components are centered on vertices, edge centers, or face centers of the underlying Platonic surface.  To name each configuration, we mark each digit $m$ among $2, 3$, and $n$ with a prime $'$ if the corresponding junction of $m$ mirror planes is the center of a boundary component.  There is some redundancy in the enumeration: several of the $*233$ configurations coincide due to the repeated digit 3, and several others admit additional symmetries making them the same as certain $*234$ configurations.

For any dihedral symmetry group $*22n$, the configurations $[2'2n']=[22'n']$ in Table \ref{Tab1} correspond to surfaces with $n+2$ boundary components, $n$ equally distributed around an equator, and two more centered at the poles. 

In each potential case, $\pi_G(\Sigma)$ is a polygon with either four, five, or six sides---corresponding in turn to the cases where the marking in Table \ref{Tab1} consists of one, two, or three primes---and there are distinct labels corresponding to each of the three types of mirror planes.  In the quadrilateral case, the three mirror edges and the $\partial \Sigma$ edge appear in some cyclic order.  In the pentagonal case, the two $\partial \Sigma$ edges are separated on one side by a mirror edge and on the other side by two mirror edges.  In the hexagonal case, the three mirror and $\partial \Sigma$ edges alternate in pairs.  In each of these potential cases, the labeling is irreducible, so Theorem \ref{Tgroup} applies to show $\firsteigen = \Coord$.

Several authors \cite{Ketover, GL, KOO, Schulz} have studied genus zero free boundary minimal surfaces with Platonic symmetries.  Ketover \cite[Theorem 6.1]{Ketover} constructed one example with each symmetry group; his examples have respectively $4, 6$, and $12$ boundary components. Girouard-Lagac{\'e} conjectured \cite[Remark 1.14]{GL} the existence of additional examples, one with $*234$ symmetry and $8$ boundary components, and two more with $*235$ symmetry with 20 and 32 boundary components.  For all but the $*235$ example with 32 boundary components, they noted that \cite[Theorem 5]{McGrath} shows $\sigma_1 = 1$, and further proved  \cite[Theorem 1.13]{GL} that $\sigma_1 = 1$ for the remaining example by
adapting arguments in \cite{McGrath}.  Finally, we note that Mario Schulz \cite{Schulz} has found numerical evidence for embedded examples for each of the types in Table \ref{Tab1}.

\begin{remark}
Recently, Kapouleas-Zou \cite{KapZou} have constructed families of genus zero free boundary minimal surfaces with $kn$ boundary components (for all sufficiently large $k$ and sufficiently large $n$ in terms of $k$) dihedral symmetry group $*22n$, with $n$ boundary components symmetrically arranged on each of $k$ parallel circles of $\partial \B^3$.  One can check that these surfaces satisfy the hypotheses of Theorem \ref{Tgroup} (recall Remark \ref{Rirred}), so that $\firsteigen = \Coord$.
\end{remark}

\subsection*{Doublings of the equatorial sphere in $\Sph^3$}
$\phantom{ab}$

Kapouleas and the second author constructed \cite{McGrathKap1} doublings of the equatorial sphere
\begin{align*}
\Spheq : = \{ (x_1, x_2, x_3, x_4)\in \Sph^3 \subset \R^4 : x_4 = 0\}.
\end{align*}
  We show that any surface consistent with the symmetries of any of those constructions satisfies the hypotheses of Theorem \ref{Tsphere} and hence has $\Ecal_{\lambda_1} = \Coord$. 

The doublings in \cite{McGrathKap1} depend on discrete parameters $\kcir, m \in \N$, with $\kcir\geq 2$ arbitrary, and $m$ large in terms of $\kcir$.  Each surface constructed resembles two copies of the equatorial sphere $\Sph^2$, joined by catenoidal bridges, with $m$ bridges placed among each of a collection of $\kcir$ parallel circles in maximally symmetric fashion.  The construction allows the option to place two more catenoidal bridges at the poles of $\Spheq$. 

\def\raa{6} 
\def\ang{360/20} 
\def\rot{-25} 
\begin{figure}[h]
\centering
\begin{subfigure}[]{.35\textwidth}
\centering
\begin{tikzpicture}[scale=.6] 
\draw [fill=light-gray, domain=-\ang:\ang, samples=50] plot ({\raa*sin(\x)}, {\raa*sin(\rot)*cos(\x)}); 
\foreach \i in {-1, 1} 
{
\draw [fill=light-gray, domain=0:90, samples=100] plot ({(\i*\raa*sin(\ang)*sin(\x)}, {\raa*cos(\x)*cos(\rot)+\raa*sin(\rot)*sin(\x)*cos(\ang)}); 
}
\draw [fill=light-gray, draw=light-gray] (0, {\raa*cos(\rot)})--({\raa*sin(\ang)}, {\raa*sin(\rot)*cos(\ang)})--({-\raa*sin(\ang)}, {\raa*sin(\rot)*cos(\ang)})--(0, {\raa*cos(\rot)}); 
\draw[dashed] (0, {\raa*sin(\rot)})--(0, {\raa*cos(\rot)}); 
\foreach \i in {0, 1, 2, 3, 4, 5} 
{
\draw [fill=white] (0, {\i*\raa/5+\raa*sin(\rot)/2-\raa/10}) circle ({\raa/35});
}
\end{tikzpicture}
\end{subfigure}
\begin{subfigure}{.35\textwidth}
\centering
\begin{tikzpicture}[scale=.6] 
\draw [fill=light-gray, domain=-\ang:\ang, samples=50] plot ({\raa*sin(\x)}, {\raa*sin(\rot)*cos(\x)}); 
\foreach \i in {-1, 1} 
{
\draw [fill=light-gray, domain=0:90, samples=100] plot ({(\i*\raa*sin(\ang)*sin(\x)}, {\raa*cos(\x)*cos(\rot)+\raa*sin(\rot)*sin(\x)*cos(\ang)}); 
}
\draw [fill=light-gray, draw=light-gray] (0, {\raa*cos(\rot)})--({\raa*sin(\ang)}, {\raa*sin(\rot)*cos(\ang)})--({-\raa*sin(\ang)}, {\raa*sin(\rot)*cos(\ang)})--(0, {\raa*cos(\rot)}); 
\draw[dashed] (0, {\raa*sin(\rot)})--(0, {\raa*cos(\rot)}); 
\draw [fill=white, domain=-2:182, samples=50] plot ({\raa/35*cos(\x)}, {\raa/35*sin(\x)+\raa*sin(\rot)}); 
\foreach \i in {1, 2, 3, 4, 5}
{
\draw [fill=white] (0, {\i*\raa/5+\raa*sin(\rot)}) circle ({\raa/35});
}
\draw [draw=white,thick] ({-\raa/35}, {\raa*sin(\rot)})--({\raa/35}, {\raa*sin(\rot)});
\end{tikzpicture}
\end{subfigure}
\caption{Schematics of fundamental domains for doublings of $\Spheq$.  For clarity, each picture depicts two fundamental domains, bisected by a (dotted) mirror plane.  The middle circles denote catenoidal waists, and the boundary edges are contained in mirror planes of symmetry.  The left example has $\kcir=12$ and the right has $\kcir =11$.  
}
\end{figure}
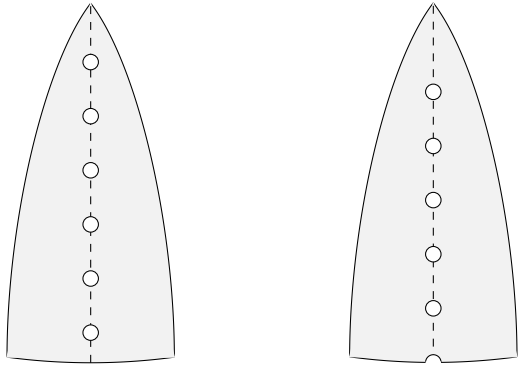

The symmetry group $G$ of the constructions (see \cite[Def. 2.8]{McGrathKap1}) contains $*22m$ as an index-two subgroup and additionally contains the reflection $\Rcap_{\{x_4 =0\}}$, where by convention in this discussion we embed $\R^3$ in $\R^4$ by $(x_1, x_2, x_3)\mapsto (x_1, x_2, x_3, 0)$.

The orbifold $\pi_G(\Sigma)$ is homeomorphic to a fundamental domain $D$ on $\Sigma$ bounded by mirror planes $P_1, P_2, P_3$, and $P_4$.  When there are no bridges at the poles, its boundary is a $(3+\kcir)$-gon consisting of edges, $\kcir$ of which alternate between $P_4$ and $P_1$, and (in order) three more edges contained  in $P_1, P_2$, and $P_3$.  When there are bridges at the poles, its boundary is a $(4+\kcir)$-gon, and the sequence of edges alternating between $P_4$ and $P_1$ contains one additional edge in $P_4$. 
 
In each of the preceding cases Assumption \ref{Agsph} holds and the labeling is irreducible by Remark \ref{Rirred}, so Theorem \ref{Tsphere} applies to show $\Ecal_{\lambda_1} = \Coord$.

\subsection*{Reducible labelings}

\begin{remark}[Surfaces in $\B^3$]
\label{Rred}
Theorem \ref{Tgroup} does not apply to surfaces with the symmetries of desingularizations of the critical catenoid and equatorial disk \cite{KapLi}, or of doublings of the critical catenoid \cite{LDG}: although such surfaces satisfy \ref{Ag}(a), the associated labelings are reducible.

In more detail,  if $\Sigma$ is a desingularization of the critical catenoid and disk, by work in \cite{KapLi} it follows that Assumption \ref{Ag}(a) holds for $G = 2*m$ and some $m\geq 3$; in particular $\pi_G(\Sigma)$ is homeomorphic to a fundamental domain $D$ on $\Sigma$ bounded by mirror planes $P_1, P_2$, a reflection line $L$, and $\partial \B^3$.  Since $\Rcap_L$ exchanges $P_1$ and $P_2$, the labeling is the same as that of the pentagon in Figure \ref{F1}, where $a = \pi_G(\partial \Sigma)$, $b = \pi_G(P_1) = \pi_G(P_2)$, $c = \pi_{G}(L)$, and is in particular reducible.

If $\Sigma$ is a doubling of the critical catenoid, by work in \cite{LDG}, it follows that Assumption \ref{Ag}(a) holds for $G = *22m$ and some $m\geq 3$; in particular, $\pi_G(\Sigma)$ is homeomorphic to a fundamental domain (an octagon) $D$ on $\Sigma$ bounded by mirror planes $P_1, P_2, P_3$ and $\partial \B^3$.  The labeling is the same as that of the octagon in Figure \ref{F1}, where $a = \pi_G (\partial \Sigma)$ , $b = \pi_G(P_2)$, $c = \pi_G(P_1)$, $d = \pi_G(P_3)$ and is reducible.
\end{remark}

\begin{remark}[Surfaces in $\Sph^3$]
\label{Rredsph}
Theorem \ref{Tsphere} does not apply to surfaces with the symmetries of doublings of the Clifford torus \cite{KapYang, Wiygul}; although such surfaces satisfy \ref{Agsph}(a), the associated labelings are reducible.  In particular, if $\Sigma$ is a doubling of the Clifford torus, $\pi_G(\Sigma)$ is homeomorphic to a fundamental domain (a hexagon) $D$ on $\Sigma$ bounded by mirror planes $P_1, P_2, P_3,$ and $P_4$.  The labeling is reducible: one cyclic ordering of the corresponding planes is $P_1, P_2, P_1, P_3, P_4, P_3$. 
\end{remark}

\end{document}